\documentclass[12pt]{amsart}

\usepackage{graphicx}					
\usepackage{amssymb}
\usepackage{mathtools}
\usepackage{xcolor}
\usepackage{mathabx}
\usepackage{amsmath}
\usepackage{amssymb}
\usepackage{amsthm}
\usepackage{url}
\usepackage{tkz-euclide}

\newtheorem{theorem}{Theorem}[section]
\newtheorem{lemma}[theorem]{Lemma}
\newtheorem{corollary}[theorem]{Corollary}

\newtheorem{claim}{Claim}[theorem]

\newtheorem{conjecture}[theorem]{Conjecture}
\newtheorem{question}[theorem]{Question}

\newtheorem*{lund}{Lund's Theorem}
\newtheorem*{beck}{Beck's Theorem}

\newtheorem*{Do}{Do's Theorem}
\newtheorem*{blah}{Theorem}

\newcommand{\del}{\setminus}
\newcommand{\con}{/}

\DeclareMathOperator{\cl}{cl}

\DeclareMathOperator{\DG}{DG}

\newcommand{\bF}{\mathbb{F}}
\newcommand{\bR}{\mathbb{R}}
\newcommand{\bZ}{\mathbb{Z}}
\newcommand{\bC}{\mathbb{C}}

\newcommand{\cL}{\mathcal{L}}
\newcommand{\cF}{\mathcal{F}}
\newcommand{\cH}{\mathcal{H}}
\newcommand{\cT}{\mathcal{T}}
\newcommand{\cP}{\mathcal{P}}
\newcommand{\cW}{\mathcal{W}}
\newcommand{\cI}{\mathcal{I}}

\begin{document}

\title{Average plane-size in complex-representable matroids}

\author{Rutger Campbell}
\address{Discrete Mathematics Group, Institute for Basic Science (IBS), Daejeon, Republic of Korea} 
\author{Jim Geelen}
\address{Department of Combinatorics and Optimization, University of Waterloo, Waterloo, Canada} \author{Matthew E. Kroeker}
\address{Department of Combinatorics and Optimization, University of Waterloo, Waterloo, Canada} 

\thanks{This research was partially supported by the Institute for Basic Science [IBS-R029-C1], by grants from Office of Naval Research [N00014-10-1-0851] and NSERC [203110-2011], and by an NSERC Postgraduate Scholarship [Application No. PGSD3 - 547508 - 2020]}

\begin{abstract}
Melchior's inequality implies that the average line-length in a simple, rank-$3$, real-representable matroid is less than $3$. A similar result holds for complex-representable matroids, using Hirzebruch's inequality, but with a weaker bound of $4$.
We show that the average plane-size in a simple, rank-$4$, complex-representable matroid is bounded above by an absolute constant, unless the matroid is the direct-sum of two lines.
We also prove that, for any integer $k$, in complex-representable matroids with rank at least $2k-1$, the average size of a rank-$k$ flat is bounded above by a constant depending only on $k$. Finally, we prove that, for any integer $r\ge 2$, the average flat-size in rank-$r$ complex-representable matroids is bounded above by a constant depending only on $r$.
We obtain our results using a theorem, due to Ben Lund, that gives a good estimate on the number of rank-$k$ flats in a complex-representable matroid.
\end{abstract}

\maketitle

\sloppy

\section{Introduction}

Melchior's inequality \cite{Melchior} implies that, given a finite set $X$ of points in $\bR^d$, not all on a line, the average length of a spanned line is less than $3$.  Here, a line is {\em spanned} if it hits at least two points of $X$, and the {\em length} of the line is the number of points of $X$ hit by the line. We consider analogous problems for spanned planes as well as higher dimensional flats. We also consider points in $\bC^d$ instead of $\bR^d$. There is an analogue of Melchior's result known for complex-space, due to Hirzebruch \cite{Hirzebruch}, which implies that the average length of a spanned line in a finite set of points in $\bC^d$, not all on a line, is less than $4$.

Even though we are intrinsically concerned with finite sets of points in $\bR^d$ or $\bC^d$, we find it convenient to state our results and proofs using the language of matroid theory, which we summarize in Appendix~A. For example, in this language, Melchior's result  states that the average line-length in a simple real-representable matroid with rank at least $3$ is less than $3$. Note that, by truncation, it suffices to prove this result when the rank is equal to $3$. More generally, if we are interested in the average size of rank-$k$ flats in a matroid $M$, then we may as well truncate $M$ to rank-$(k+1)$. Therefore, we are implicitly interested the average size of hyperplanes. 

Given Hirzebruch's result, one might hope that average plane-size is bounded in rank-$4$ complex-representable matroids, but this is not the case.
In the direct sum of two lines of length $a$, that is $U_{2,a}\oplus U_{2,a}$, all planes have size $a+1$, so we can make the average plane-size arbitrarily large by choosing $a$ appropriately. We prove however that, for a rank-$4$ complex-representable matroid $M$
the average plane-size is bounded unless $M$ is the direct sum of two lines.
\begin{theorem}\label{main_planes}
In a simple complex-representable matroid $M$ with rank at least $4$, the average plane-size is bounded above by an absolute constant, unless $M$ is the direct sum of two lines.
\end{theorem}
We prove our result by strengthening a theorem of Lund that estimates, up to a multiplicative error, the number of planes in a complex representable matroid. We do not explicitly determine the constant in Theorem~\ref{main_planes}, and while our proof does give a computable bound it is certainly far from best possible (Section~\ref{spec} contains further discussion about the constants).

We would like to find an analogue of Theorem~\ref{main_planes} for the average-size of rank-$k$ flats in simple complex-representable matroids with rank at least $k+1$; we don't yet have a satisfactory conjecture, but we do have the following interesting partial result. 
\begin{theorem}\label{main_flats}
For an integer $k\ge 2$, in a simple complex-representable matroid with rank at least $2k-1$, the average size of a rank-$k$ flat is bounded above by a constant depending only on $k$.
\end{theorem}
For $k=3$, Theorem~\ref{main_flats} is a corollary of Theorem~\ref{main_planes}.

We also prove the following result concerning average flat-size.
\begin{theorem}\label{thm_flats}
For each integer $r$,  the average flat-size in a simple rank-$r$ complex-representable matroid is bounded above by a constant depending only upon $r$.
\end{theorem}
For $r=4$, Theorem~\ref{thm_flats} is a corollary of Theorem~\ref{main_planes} combined with Hirzebruch's result. Moreover, Theorem~\ref{main_flats} follows from Theorem~\ref{thm_flats} using
results of Huh and Wang on ``top-heavyness''; see Section~\ref{top-heavy}. Nevertheless, we will give a direct proof of both Theorems~\ref{main_flats} and~\ref{thm_flats}.

\section{Speculation}\label{spec}

In this section, we discuss what we know concerning lower bounds for the constants in our results, and pose several conjectures due to the second author. The conjectures essentially say that, asymptotically, the worst-case behaviour is exhibited by Dowling Geometries; these are defined in Appendix~B. We also provide some justification for the claims about the average flat-sizes in the appendix; see Lemma~\ref{weakbounds}. We typically consider Dowling Geometries $\DG(r,\bZ_t)$ where $r$ is fixed and $t$ goes to infinity. These matroids are complex-representable but not real-representable, so the conjectures may not be best-possible for real-representable matroids.

Throughout this section $M$ denotes a simple $n$-element complex-representable matroid with rank at least $3$.
We let $\cF_k(M)$ denote the set of rank-$k$ flats of $M$ and denote $\cF_2(M)$ and $\cF_3(M)$ by $\cL(M)$ and $\cP(M)$ respectively.
We will drop the specific reference to $M$ when there is little potential for ambiguity.

Starting with lines in real-representable matroids, Melchior's Inequality can be written as:
$$ \sum_{L\in \cL}(|L|-3) \le - 3.$$
Equality is attained when $M$ has all but one point on a single line. Thus the average line-length can get arbitrarily close to $3$.

In the complex-representable case Hirzebruch showed that, unless all but one point of $M$ lie on a single line,
$$ \sum_{L\in \cL}(|L|-4) \le - (n+m_2),$$
where $m_2$ denotes the number of $2$-point lines. This implies that the average line-length is less than $4$. According to Pokora~\cite{Pokora}, there is only one matroid that gives equality, namely $\DG(3,\bZ_3)\del\{v_1,v_2,v_3\}$, which is isomorphic to the ternary affine plane. However the average line-length in that example is only $3$. For $t>4$, the average line-length in $\DG(3,\bZ_t)\del\{v_1,v_2,v_3\}$ is strictly larger than $3$, but it tends to $3$ when $t$ is large. Peter Nelson (personal communication) asked whether $4$ is in fact best possible. The best example that we know of is the dual configuration of Wiman's line arrangement, which has 45 elements and an average line-length of $\frac{240}{67}\approx 3.58.$
\begin{conjecture}\label{conj_lines}
The average line-length in a simple $n$-element, rank-$3$, complex-representable matroid is at most $3+o_n(1)$.
\end{conjecture}
Here the $o_n(1)$ term denotes a function, depending on $n$, that tends to $0$ as $n$ goes to infinity. 

The matroid $\DG(4,\bZ_2)$ is real-representable and has average plane-size equal to $6$. For $t>2$, the average plane-size of
$\DG(4,\bZ_t)$ is strictly larger than $6$, but it tends to $6$ for large $t$.
\begin{conjecture}\label{conj_planes}
The average plane-size in a simple $n$-element, rank-$4$, complex-representable matroid, that is not the union of two lines, is at most $6+o_n(1)$.
\end{conjecture}

For fixed $r\ge 2$, the average flat-size in $\DG(r,\bZ_t)$ tends to ${r\choose 2}$, contrary to a conjecture made in an earlier version of this paper.
\begin{conjecture}\label{conj_flats1}
For any fixed integer $r\ge 3$,
the average size of a flat in a simple $n$-element, rank-$r$ complex-representable matroid is at most ${r\choose 2}+o_n(1)$.
\end{conjecture}
Note that, since we are treating $r$ as a fixed constant here, the function indicated by the $o_n(1)$ term can depend on both $r$ and $n$.

The next conjecture is a strengthening of Conjecture~\ref{conj_flats1}.
\begin{conjecture}\label{conj_flats2}
In a simple complex-representable matroid $M$, the average of $\frac{|F|}{{r(F)+1\choose 2}}$, taken over all flats of $F$ of $M$, is at most $1+o_n(1)$.
\end{conjecture}

For real-representable matroids we do not know of constructions whose average flat-size is super-linear in the rank.
\begin{question}
Is the average size of a flat in a simple $n$-element, rank-$r$ real-representable matroid at most $O(r)$?
\end{question}

The final two conjectures are of a different nature; they concern generalizations from complex-representable matroids to arbitrary matroids. We state the conjectures for lines and planes, but they should generalize to higher dimensional flats in matroids with sufficiently high rank.
\begin{conjecture}\label{conj_matroids}
For every real number $\alpha_2$, there is a real number $\alpha_3$ such that, for any simple matroid $M$ with rank at least $6$, if the average line-length of $M$ is at most $\alpha_2$, then the average plane-size is at most $\alpha_3$.
\end{conjecture}
Ben Lund (personal communication) showed that the bound of $6$, in Conjecture~\ref{conj_matroids}, is best-possible (refuting a conjecture in an earlier version of this paper). He constructs rank-$5$ matroids by taking the direct sum of a projective plane of order $p$ and an $n$-point line. If $n$ is much larger than $p$, then the average line-length is less than $3$, but all planes have size at least $p$.

Surely a converse to Conjecture~\ref{conj_matroids} must hold; that is, bounding average plane-size in a matroid should also bound average line-length. Moreover, one would expect that this has an easy proof, but we have not found it. 
\begin{conjecture}\label{shameful}
For every real number $\beta_3>0$, there is a real number $\beta_2>0$ such that, for any simple matroid $M$ with rank at least $4$, if the average plane-size of $M$ is at most $\beta_3$, then the average line-length is at most $\beta_2$.
\end{conjecture}

\section{Average line-length}

Recall that Hirzebruch's Inequality implies that the average line-length in a {complex-representable} matroid is less than $4$. As an illustration of the methods used later in this paper,  in this section we reprove that result, with a weaker bound of $6$ on the average line-length. Our proof uses the following result due to De Zeeuw \cite{DeZeeuw} which is a variation of Beck's Theorem (see Section~3).
\begin{theorem}\label{DeZeeuw}
If $M$ is a simple, rank-$3$, complex-representable matroid with $n$ points, then either
$M$ has a line of size at least $0.9 n$, or 
$M$ has at least $\frac 1{12}n^2$ lines.
\end{theorem}

Note that, if the length of a longest line in a simple, $n$-element, rank-$3$ matroid  is $g_2$, then 
\begin{align*}
|\cL|&\le 1+ g_2(n-g_2) + {n- g_2\choose 2}\\
& \le  n(n-g_2) +\frac{1}{2}(n-g_2)^2 \\
&\le \frac 3 2 n(n-g_2).
\end{align*}

The following corollary to Theorem~\ref{DeZeeuw} shows that this upper bound is of the right order of magnitude in complex-representable matroids.
\begin{lemma}\label{lem:lines}
For any simple, $n$-element,  rank-$3$, complex-representable matroid $M$, we have
$$|\cL(M)|\ge \frac 1{12} n (n-g_2),$$
where $g_2$ denotes the length of a longest line in $M$.
\end{lemma}

\begin{proof}  Let $L$ be a longest line in $M$ and let $C=E(M)\setminus L$.
We may assume $|\cL|<\frac 1{12} n^2$ since otherwise the result holds. Then, by Theorem~\ref{DeZeeuw},  we have $|L|\ge  0.9 n$ and hence $|C|\le 0.1 n$. For each point $e\in C$, consider the lines that contain $e$ and a point of $L$ but no other points of $C$.
As these lines have exactly two points, they are uniquely determined by the point in $C$ and the point in $L$. Thus
$$ |\cL| \ge  (|L|-|C|)\times |C| \ge  0.8 n (n-g_2),$$
as required.
\end{proof}

Now we can use Lemma~\ref{lem:lines} to bound average line-length.
\begin{theorem}\label{thm:lines}
The average line-length in a simple, rank-$3$, complex-representable matroid is at most $6$.
\end{theorem}

\begin{proof}
Let $M$ be a simple, rank-$3$, complex-representable matroid with $n$ points, let $L^*$ be a longest line in $M$, and let $g_2=|L^*|$. By Lemma~\ref{lem:lines},
we have $|\cL|\ge \frac 1{12} n (n-g_2)$. Now let $\cT$ denote the collection of all ordered pairs
$(e,f)$ of distinct elements with $e\in E(M)$ and $f\in E(M)\setminus L^*$.
Thus
$$ |\cT| \le n(n-g_2) \le 12 |\cL|.$$
Each pair in $\cT$ spans a unique line, and each line $L$, other than $L^*$, contains
at least $(|L|-1)^2$ distinct pairs from $\cT$ (since we can choose $f\in L\setminus L^*$ arbitrarily and then choose $e\in L\setminus\{f\}$). 
Moreover 
$$ (|L|-1)^2 = 12(|L|-4) + (|L|-7)^2 \ge 12(|L|-4).$$
Therefore
$$ \sum_{L\in \cL\setminus\{L^*\}} |L|-4 \le \frac 1{12} \cT \le |\cL|.$$
Thus the average length of a line in $\cL\setminus\{L^*\}$ is at most $4+1=5$.
By considering the lines through an element off of the line $L^*$, we have $|\cL|>|L^*|$.
Therefore the line $L^*$ contributes at most $1$ to the average line-length.
\end{proof}

\section{Bounding the number of rank-$k$ flats}

In Section~2, we saw that bounding average line-length followed easily from a good lower-bound on the number of lines. Because we are bounding the average by a constant, it is important that the lower-bound is of the right order of magnitude. In this section, we consider the problem of finding good lower-bounds on the
number of rank-$k$ flats in real- and complex-representable matroids. Fortunately, this is a well-studied problem. Much of that work is founded on a result of Szemer\'edi and Trotter \cite{SzemerediTrotter} that bounds point-line incidences in real-representable matroids; their result was extended to complex-representable matroids by T\' oth \cite{Toth} and Zahl \cite{Zahl}, independently. 

The number of rank-$k$ flats in an $n$-element matroid is at most $n^k$. We begin by considering when this bound is of the right order-of-magnitude. For lines, this was solved by Beck \cite{Beck}; he stated the result for real-representable matroids, but his result can be obtained as a corollary to the Szemer\'edi-Trotter Theorem and therefore also works in the complex-representable setting.
\begin{beck}
For each real number $\epsilon>0$, there exists a real number $\gamma_2(\epsilon)$ such that, if
$M$ is a simple, complex-representable matroid with $n$ points, then either
$M$ has at least $\gamma_2(\epsilon)n^2$ lines, or $M$ has a set of elements with rank at most $2$ and size at least  $(1-\epsilon) n$.
\end{beck}

The right generalization of Beck's Theorem to rank-$k$ flats was found by Do \cite{Do}, using Lund's notion of ``essential dimension" \cite{Lund}. A similar notion had previously been introduced by Lov\' asz~\cite[Theorem ~2.3]{Lovasz} who, together with Yemini, found applications to bar-and-joint rigidity problems~\cite{LovYem}.

For an integer $k\ge 2$, we say that a simple matroid $M$ is {\em $k$-degenerate} if either it has rank at most $1$ or there there exist flats $(F_1,\ldots,F_t)$ in $M$, each with rank at least two, such that $F_1\cup\cdots\cup F_t=E(M)$ and
\begin{equation}\label{degenerate}
(r(F_1)-1)+\cdots+(r(F_t)-1) < k.
\end{equation}
So a matroid is $2$-degenerate if and only if it has rank at most $2$ and is $3$-degenerate if and only if it either has rank at most $3$ or it is the union of two skew lines.

Before stating Do's Theorem, we consider some of the properties of $k$-degenerate matroids. Let $M$ be a simple matroid and let $(F_1,\ldots,F_t)$ be flats, each with rank at least two, that cover $M$ and satisfy (\ref{degenerate}). We claim that:
\begin{itemize}
\item[(a)] for $1\le i<j\le t$, if $F_i$ and $F_j$ are not skew then $r(F_i\cup F_j)-1 \le (r(F_i)-1)+(r(F_j)-1)$,
\item[(b)] each rank-$k$ flat in $M$ spans one of the flats $F_1,\ldots,F_t$,
\item[(c)] $t< k$, and
\item[(d)] $r(M)\le 2(k-1)$.
\end{itemize}
Claim $(a)$ is immediate.
For $(b)$, consider a basis $B$ of a rank-$k$ flat $F$. By the pigeon-hole principle, there exists 
$i\in\{1,\ldots,t\}$ such that $F_i$ contains at least $r(F_i)$ elements of $B$. But then $F_i$ is spanned by $F$, as required. Claim $(c)$ follows from (\ref{degenerate}) since each of that terms $r(F_i)-1$ is at least $1$. For $(d)$ we have
\begin{align*}
k &\ge (r(F_1)-1)+\cdots+(r(F_t)-1) + 1\\
&\ge r(F_1\cup\cdots\cup F_t) - (t-1) \\
&\ge r(M) - (k-2),
\end{align*}
which gives $r(M)\le 2(k-1)$, as required.

By $(a)$, we may assume that the flats $(F_1,\ldots,F_t)$ are pair-wise skew and therefore that $(F_1,\ldots,F_t)$ is a partition of $E(M)$.
By $(b)$, the number of rank-$k$ flats in a $k$-degenerate matroid is at most $tn^{k-2}< kn^{k-2}$, which is $o(n^k)$, for fixed $k$.

Lund defines the {\em essential dimension} of a matroid to be the maximum $k$ such that $M$ is not $(k+1)$-degenerate. (The ``$+1$" here is due to the fact that geometers consider the dimension of flats whereas we are using their rank.)
Note that the essential dimension is at most $r(M)-1$
and, by $(d)$, is greater than $\frac 1 2 r(M)$. 

We say that a set $X$ of elements is {\em $k$-degenerate} if $M|X$ is $k$-degenerate.
\begin{Do}[\cite{Do}]
For each real number $\epsilon>0$ and integer $k\ge 2$, there exists a real number $\gamma_k(\epsilon)$ such that, if
$M$ is a simple, complex-representable matroid with $n$ points, then either
the number of rank-$k$ flats in $M$ is at least $\gamma_k(\epsilon)n^k$, or $M$ has a $k$-degenerate set of elements of size at least  $(1-\epsilon) n$.
\end{Do}
 
 We let $g_k(M)$, or simply $g_k$, denote the size of the largest $k$-degenerate set in $M$. For lines we saw that $|\cL| = \Theta(n(n-g_2)).$ For rank-$k$ flats  Lund showed that $|\cF_k|$ is $O(n(n-g_2)\cdots(n-g_k))$.
Lund~\cite{Lund} also showed that, for matroids that are not too close to being $k$-degenerate, this bound is of the right order of magnitude.
\begin{lund} 
For each integer $k\ge 2$ there is a real number $\rho_k>0$ and an integer $c_k\ge 0$ such that, if
$M$ is a simple, $n$-element, complex-representable matroid with $n-g_k>c_k$, then
$$ |\cF_k(M)| \ge \rho_k n(n-g_2)(n-g_3)\cdots (n-g_k).$$
\end{lund}

The condition that $n-g_k>c_k$ is necessary.
Consider the rank-$6$ matroid $M$ with $n=3(a+2)$ elements depicted in Figure~\ref{fig1}. We construct $M$ by $2$-summing three $(a+1)$-point lines onto the three points of a triangle of $M(K_4)$, where $a$ can be chosen arbitrarily large. The figure depicts an example with $a=3$; the matroid $M(K_4)$ is depicted by the six coplanar blue points. One can check that $g_2=a+1$, $g_3=2(a+1)$, $g_4=3(a+1)$, and $g_5=n-1$. Each hyperplane is forced to contain one of the three $(a+1)$-point lines, so the number of hyperplanes is only $O(n^2)$ but $n(n-g_2)\cdots(n-g_5)$ is $\Theta(n^3)$. This example is also of interest since the average hyperplane size is arbitrarily large.

\begin{figure}
\begin{tikzpicture}[scale=0.6]

    \node[fill=blue, circle, scale=0.5] at (2,3)(a){};
    \node[fill=blue, circle, scale=0.5] at (5,3)(b){};
    \node[fill=blue, circle, scale=0.5] at (8,3)(c){};
    
    \node[fill=blue, circle, scale=0.5] at (6,6)(d){};
    \node[fill=blue, circle, scale=0.5] at (4,4.5)(e){};

    \tkzInterLL(e,c)(b,d)       \tkzGetPoint{f}
    \node[fill=blue, circle, scale=0.5] at (f){};

    \draw[blue, thick] (a) -- (c) -- (e);
    \draw[blue, thick] (a) -- (d) -- (b);

    \node[fill=red, circle, scale=0.5] at (0,0)(x1){};
    \node[fill=red, circle, scale=0.5] at (.66,1)(x2){};
    \node[fill=red, circle, scale=0.5] at (1.34,2)(x3){};
    \draw[red, thick] (x1) -- (a);

    \node[fill=red, circle, scale=0.5] at (0,0)(a1){};
    \node[fill=red, circle, scale=0.5] at (.66,1)(a2){};
    \node[fill=red, circle, scale=0.5] at (1.34,2)(a3){};
    \draw[red, thick] (a1) -- (a);

    \node[fill=red, circle, scale=0.5] at (6,0)(b1){};
    \node[fill=red, circle, scale=0.5] at (5.67,1)(b2){};
    \node[fill=red, circle, scale=0.5] at (5.33,2)(b3){};
    \draw[red, thick] (b1) -- (b);

    \node[fill=red, circle, scale=0.5] at (10,0)(c1){};
    \node[fill=red, circle, scale=0.5] at (9.34,1)(c2){};
    \node[fill=red, circle, scale=0.5] at (8.67,2)(c3){};
    \draw[red, thick] (c1) -- (c);
\end{tikzpicture}
\caption{A rank-$6$ matroid}\label{fig1}
\end{figure}
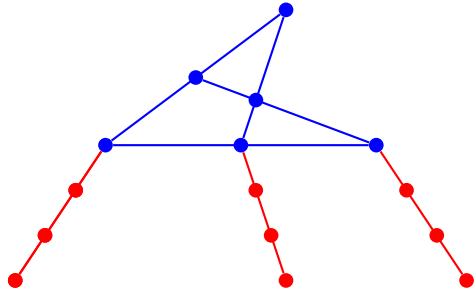

For planes however, we improve Lund's Theorem by proving that $c_3=0$. For ease of application, we prove a stronger-looking version of the result. The two versions are in fact equivalent, which follows from Lemma~\ref{strat1}.
\begin{theorem} \label{lund_planes}
There exists a real number $\varrho_3>0$ such that, if
$M$ is a simple, $n$-element, complex-representable matroid that is not $3$-degenerate, $X_3$ is a largest $3$-degenerate set in $M$, and $X_2$ is a largest $2$-degenerate subset of
$X_3$, then
$$ |\cP(M)| \ge \varrho_3 n(n-|X_2|)(n-|X_3|).$$
\end{theorem}

\begin{proof} 
We prove the result for 
\begin{align*}
     \varrho_3 = \min \left\{ \gamma_{3}(\epsilon), \, \frac{1}{120},\, \left( \gamma_{2}(1/4) -  \frac{\epsilon}{1-\epsilon} \right)(1- \epsilon)^{2} \right\},
\end{align*}
where
 $$0< \epsilon < \min\left\{ \frac 1{10}, \,\gamma_{2}(1/4)/(1+\gamma_{2}(1/4))\right\},$$ 
 and where $\gamma_2(1/4)$ and $\gamma_{3}(\epsilon)$ are given by Beck's Theorem and Do's Theorem respectively.
Let $Y_{3}:=E(M) \setminus X_{3}$ and $Y_{2}:=X_{3} \setminus X_{2}$. We may assume that $|\mathcal{P}(M)| < \gamma_{3}(\epsilon)n^{3}$. So by Do's Theorem, $|Y_{3}| < \epsilon n$.
Since $M$ is not $3$-degenerate, the rank of $M$ is at least $4$ and $M$ is not the union of two lines. By truncation, we may assume that $r(M)=4$. Note that $X_3$ is either a plane or the union of two skew lines and that $X_2$ is the longest line in $M|X_3$. 

{\flushleft \textbf{Case 1.} $|X_{2}| \geq \frac{3}{5}n$.}

Let $N$ be obtained from $M$ by principally truncating the line $X_{2}$ to a point. By choice of $(X_2,X_3)$,  the set $Y_{2}$ is a longest line in $N \setminus X_{2}$. By Lemma \ref{lem:lines}, we have $|\mathcal{L}(N \setminus X_{2})| \geq \frac 1{12}(n-|X_{2}|)(n-|X_{3}|)$. 
Note that at most $\frac{1}{2}|E(N \setminus X_{2})|$ of the lines in $N \setminus X_{2}$ span the point $X_{2}$ in $N$. Since $M$ is not the direct-sum of two lines, $N \setminus X_{2}$ has rank $3$ and therefore, by the De Bruijn-Erd\H{o}s Theorem (see Theorem~\ref{DB-E} or~\cite[Theorem 1]{DeBruijn}), has at least $|E(N \setminus X_{2})|$ lines. 
Thus at most half the lines of $N \setminus X_{2}$ span $X_{2}$ in $N$. It follows that there are at least $\frac{1}{24}(n-|X_{2}|)(n-|X_{3}|)$ lines in $M$ which are skew to $X_{2}$. Let $L$ be such a line, and let $X_{2}'$ be the set of points $e \in X_{2}$ such that $L \cup \{e\}$ is a plane in $M$. Then $|X_{2}'| \geq |X_{2}| - |E(M) \setminus X_{2}| \geq \frac{1}{5}n$. Hence the number of planes in $M$ is at least $\frac{1}{120}n(n-|X_{2}|)(n-|X_{3}|)$, which gives the result in this case.

{\flushleft \textbf{Case 2.} $|X_{2}| < \frac{3}{5}n$.}

Suppose first that $X_{3}$ is a plane. Observe that, since $|X_3|\ge (1-\epsilon)n$, we have
$|X_{2}| < \frac{3}{5(1- \epsilon)}|X_{3}| < \frac{3}{4}|X_{3}|.$
 So by Beck's Theorem, $|\mathcal{L}(M|X_{3})| \geq \gamma_{2}( \frac{1}{4})|X_{3}|^{2}$. Now consider an element $e \in Y_{3}$. In $M/e$, each element of $Y_{3} \setminus \{e\}$ is spanned by at most $|X_{3}|$ distinct lines of $M|X_{3}$.  Therefore at most $|X_{3}||Y_{3}| < \frac{\epsilon}{1-\epsilon}|X_{3}|^{2}$ lines of $M|X_{3}$ span a point of $Y_{3}$ in $M/e$. So at least $\Big( \gamma_{2}(\frac{1}{4}) - \frac{\epsilon}{1-\epsilon} \Big)|X_{3}|^{2}$ lines of $M|X_{3}$ span no elements of $Y_{3}$ in $M/e$. Hence
\begin{align*}
     |\mathcal{P(M))}| &\geq \Big( \gamma_{2}(1/4) -  \frac{\epsilon}{1-\epsilon} \Big)|X_{3}|^{2}(n-|X_{3}|) \\
     &> \Big( \gamma_{2}(1/4) -  \frac{\epsilon}{1-\epsilon} \Big)(1-\epsilon)^{2}n^2(n-|X_{3}|)\\
     &> \Big( \gamma_{2}(1/4) -  \frac{\epsilon}{1-\epsilon} \Big)(1-\epsilon)^{2}n(n-|X_{2}|)(n-|X_{3}|),
\end{align*}
as desired.

Now suppose that $X_{3}$ is the union of two skew lines, namely $X_2$ and $Y_2$. Observe that every plane in $M|X_{3}$ exactly consists of one of the lines $X_2$ and $Y_2$ together with a point from the other. For a fixed element $e\in Y_3$  we will consider the set $\cP_e$ of $3$-point planes in $M$ that contain $e$. Note that $Y_2\cup\{e\}$ spans a plane and that 
plane contains at most one point of $X_2$. Moreover, there are at most $|Y_3|-1$ lines in $M|Y_3$ that contain $e$ and each of these lines spans at most one point of $X_2$.
Therefore there is a set $X'_2\subseteq X_2$ with $|X'_2|\ge |X_2|-|Y_3|$ such that for each $x\in X'_2$ the set $\{e,x\}$ is a line of $M$ and that line is skew to $Y_2$. 
Fix an element $x\in X'_2$. By our choice of $X'_2$, the lines $\{e,x\}$ and $Y_2$ are skew, so there are at most $|Y_3|$ planes in $M$ that contain $\{e,x\}$ as well as some other element of $Y_3$, and each of these planes intersects the line $Y_2$ in at most one point. Therefore there are at least $|Y_2|-|Y_3|$ distinct planes in $\cP_e$ that contain $\{e,x\}$.
Hence $|\cP_e|\ge  |X'_2|\times (|Y_2|-|Y_3|)$ and
$$ |\cP| \ge |Y_3|\times (|X_2|-|Y_3|)\times (|Y_2|-|Y_3|).$$
Now $|Y_3|<\epsilon n\le \frac 1 {10} n$ and $|X_2|\ge |Y_2| \ge n - |X_2|-|Y_3| \ge \frac{3}{10}n$. Therefore $ |X_2|-|Y_3|\ge \frac 1 5 n$ and $|Y_2|-|Y_3|\ge \frac 1 5 n$.
Hence
\begin{align*}
     |\mathcal{P}(M)| &  \ge |Y_3|\times (|X_2|-|Y_3|)\times (|Y_2|-|Y_3|) \\
     &\geq \frac 1{25} n^2(n-|X_3|) \\
     &\geq \frac{1}{25}n(n - |X_{2}|)(n-|X_3|),
\end{align*}
as required.
\end{proof}

\section{Average plane size}

We now use Theorem~\ref{lund_planes} to prove Theorem~\ref{main_planes}, which we restate here for convenience.
This result follows directly from Lemma~\ref{strat2}, which is proved later, but we prove this special case directly since the proof for planes is simpler.
\begin{blah}[Theorem~\ref{main_planes}]
In a simple complex-representable matroid $M$ with rank at least $4$, the average plane-size is bounded above by an absolute constant, unless $M$ is the direct sum of two lines.
\end{blah}

\begin{proof} 
Let $X_3$ be the largest set consisting of either a plane or of the union of two lines and let 
$X_2$ be the largest line in $M|X_3$. By assumption, $X_3\neq E(M)$.
Let $\cP_2$ denote the set of all planes that contain $X_2$, and let $\cP_{3}$ denote the set of all planes that contain three independent elements from $X_3$. 

There are at most $n-|X_2|$ planes that contain the line $X_2$ and each of these planes has size at most $n$. So 
\begin{equation}\label{P1}
\sum_{P \in \mathcal{P}_{2}}|P| \leq n(n-|X_2|) \le n(n-|X_2|)(n-|X_3|).
\end{equation}

We claim that:
\begin{equation}\label{P2}
\sum_{P \in \mathcal{P}_{3}\setminus\cP_2}|P| \le 2n(n-|X_2|)(n-|X_3|).
\end{equation}
Now $X_3$ is either a plane or the direct sum of two lines. In the former case,
$\cP_2=\{X_3\}$, so
$$\sum_{P \in \mathcal{P}_{3}}|P| \leq n \le n(n-|X_2|)(n-|X_3|),$$
and hence the claim holds. Therefore we may assume that $X_3$ is the union of two lines, $X_2$ and $X_3\setminus X_2$. Each plane in $\cP_3\setminus \cP_2$ contains the line $X_3\setminus X_2$ and a point of $X_2$. Moreover each point in $E(M)\setminus (X_3\setminus X_2)$ is contained in at most one of these planes. Therefore
\begin{align*}
\sum_{P \in \mathcal{P}_{3}\setminus \cP_2}|P| &\leq |X_2|\times |X_3\setminus X_2| + |E(M)\setminus (X_3\setminus X_2)| \\ &\le n(n-|X_2|) + n \\
&\le 2 n(n-|X_2|)(n-|X_3|),
\end{align*}
proving~(\ref{P2}).

Now let $\cP_1$ denote the set of all planes other than those in $\cP_2\cup \cP_3$.
Each plane $P\in\cP_1$ contains at most one element from $X_2$ and at most a line from $X_3$.
Let $\cT$ denote the set of all triples 
$(x,y,z)$ with $z\in E(M)$, $y\in E(M)\setminus X_2$, and $x\in E(M)\setminus X_3$ 
such that $\{x,y,z\}$ spans a plane. Thus 
$|\cT|\le n(n-|X_2|)(n-|X_3|)$. Note that for each basis $B$ of a plane $P\in\cP_1$, we can order
the elements of $B$ to get a triple in $\cT$. Furthermore, by basis exchange, there are 
at least $|P|-2\ge\frac 1 3|P|$ bases of $P$.
Therefore
\begin{equation}\label{P3}
 \sum_{P\in\cP_1} |P|\le 3 |\cT| = 3 n(n-|X_2|)(n-|X_3|).
\end{equation}

By inequalities~(\ref{P1}),~(\ref{P2}), and~(\ref{P3}) and Theorem~\ref{lund_planes},
$$
\sum_{P\in\cP} |P|
\le 6n(n-|X_2|)(n-|X_3|) 
\le \frac{6}{\varrho_3} |\cP|,
$$
where $\varrho_3$ is given by Theorem~\ref{lund_planes}.
So the average plane-size is at most $\frac{6}{\varrho_3}$.
\end{proof}

\section{Stratification}
A {\em $k$-stratification} of a matroid $M$ is a sequence $(X_2,\ldots,X_k)$ of sets such that
$X_i$ is $i$-degenerate and $X_2\subseteq X_3\subseteq\cdots\subseteq X_k$. A $k$-stratification
$(X_2,\ldots,X_k)$ is {\em optimal} if $X_k$ is a largest $k$-degenerate set in $M$ and for each $i\in\{2,\ldots,k-1\}$, the set $X_{i}$ is a largest 
$i$-degenerate subset of $X_{i+1}$.
\begin{lemma}\label{strat1}
If  $(X_2,\ldots,X_k)$ is an optimal $k$-stratification of an $n$-element matroid, then,
for each $i\in \{2,\ldots,k\}$,
$$ n-|X_i| \le(k+1-i) (n-g_i).$$
\end{lemma}

\begin{proof} The result is trivial when $i=k$. Suppose that $i<k$ and that 
$n-|X_{i+1}| \le (k-i) (n-g_{i+1})$. Let $Z$ be a largest $i$-degenerate set in $M$. Now $Z\cap X_{i+1}$ is also $i$-degenerate, so
$|X_i|\ge |Z\cap X_{i+1}|\ge |Z| - (n-|X_{i+1}|)$. Thus
\begin{align*}
n-|X_i| & \le (n-|Z|) + (n-|X_{i+1}|)\\
&\le (n-g_i) +(k-i) (n-g_{i+1})\\
&\le (n-g_i)+ (k-i)(n-g_i)\\
&= (k-i+1)(n-g_i),
\end{align*}
as required. 
\end{proof}

We recall that the {\em essential dimension} of a matroid $M$ is the maximum value of $k$ such that $M$ is not $k$-degenerate. Note that, an $n$-element matroid is not $k$-degenerate if and only if $n-g_k(M)>0$.
We can now prove our main technical lemma.
\begin{lemma}\label{strat2}
For an integer $k\ge 2$, if $(X_2,X_3,\ldots,X_k)$ is an optimal $k$-stratification in a simple, $n$-element complex-representable matroid $M$ with essential dimension $d$, then
$$ \sum_{F\in\cF_k} (|F|-k)  \le 2^{k(k-1)}(n-|X_1|)(n-|X_2|)\cdots (n-|X_d|),$$
where $X_{1}$ is defined to be the empty set.
\end{lemma}

\begin{proof}
For notational convenience we let $X_{k+1}=E(M)$.
 We let $\cT_i$ denote the set of $i$-element independent sets $I$ such that
$|X_j\cap I|<j$ for each $j\in \{1,2,\ldots,k\}$.

\begin{claim} \label{T}
For each $j\in\{1,\ldots,k\}$, we have
$$|\cT_j|\le (n-|X_1|)(n-|X_2|)\cdots (n-|X_j|).$$
\end{claim}

\begin{proof}[Proof of claim.]
Fix an ordering $(x_1,\ldots,x_n)$ of $E(M)$ refining the partial order $(X_2\setminus X_1,X_3\setminus X_2,\ldots,X_{k+1}\setminus X_k)$.
Consider an ordering $\{e_1,\ldots,e_j\}$ of a set $I\in \cT_j$ respecting the order 
$(x_1,\ldots,x_n)$.
For each $i\in \{1,\ldots,j\}$ we have $|X_i\cap I|<i$, so $e_i\in E(M)\setminus X_i$. Now the result follows immediately.
\end{proof} 

We let $\cH_0$ denote the set of flats $F\in \cF_k$ such that $r(F\cap X_i)< i$ for all $i\in\{2,\ldots,k\}$.
The set $\cH_0$ typically contains most of the flats in $\cF_k$ and it is easy to bound the sum of the sizes of these flats.

\begin{claim} \quad
$\sum_{F\in \cH_0}(|F|-k) \le |\cT_d|.$
\end{claim}

\begin{proof}[Proof of claim.] We may assume that $d\geq k$, since otherwise $\cH_0$ is empty.
By basis exchange, each flat $F\in \cH_0$ has at least $|F|-k$ distinct bases, each in $\cT_k$. The claim follows since rank-$k$ independent sets span unique flats.
\end{proof}

Bounding the sum of the sizes of the other flats is a bit harder.
Consider $i\in\{2,\ldots,k\}$. Since $X_i$ is $i$-degenerate, there is a partition $\cP_i$ of $X_i$ into flats, each with rank at least $2$, such that
$$ \sum_{F\in \cP_i} r(F)-1 \le i-1.$$
Note that $|\cP_i|\le i\le k$. Let $\cW$ be the collection of all subsets of $E(M)$ obtained by taking the union of sets in each non-empty subcollection of $\cP_2\cup\cdots\cup\cP_k$.
Since $|\cP_2\cup\cdots\cup\cP_k|<k(k-1)$, we have $|\cW|\le 2^{k(k-1)}-1$.

\begin{claim} 
Each flat in $F\in \cF_k\setminus \cH_0$ contains a set from $\cW$.
\end{claim}

\begin{proof}[Proof of claim.]
Since $F\not\in \cH_0$, we have $r(F\cap X_i)\ge i$ for some $i\in \{2,\ldots, k\}$. Let $I$ be an $i$-element independent subset of $F\cap X_i$.
Since $|I|>\sum_{P\in \cP_i} r(P)-1$, there is a set $P\in \cP_i$ with $|I\cap P|\ge r(P)$. Therefore $P\subseteq F$. Moreover, by construction, $P$ is contained in $\cW$.
\end{proof}

For each set $W\in \cW$, we let $\cH_W$ denote the collection of all flats in $\cF_k\setminus \cH_0$ that contain the set $W$ but that do not contain any larger set in $\cW$.
\begin{claim}\label{counting_indep}
For each $W\in \cW$ and $F\in \cH_W$, if $I$ is a maximum independent subset of $F$ in $M\con W$ and $e\in\cl(W)$, then
$I\cup\{e\}\in \cT_{a+1}$, where $a=k-r(W)$. Moreover $a+1\le d$.
\end{claim}

\begin{proof}[Proof of claim.]
Note that $|I|= r(F) -r(W) = k-r(W)=a$ and since $I$ is independent in $M\con W$, the sets $I$ and $W$ are skew in $M$, so
$I\cup\{e\}$ is an $(a+1)$-element independent set in $M$. Suppose for a contradiction that $I\cup\{e\}\not\in \cT_{a+1}$, then
there exists $i\in\{2,\ldots,k\}$ such that $|(I\cup\{e\})\cap X_i|\ge i$. Then there is a set $P\in \cP_i$ such that $|(I\cup\{e\})\cap P|\ge r(P)$.
Therefore $P \subseteq F$. But then, by definition of $\cH_{W}$, $P \subseteq W$. Since $r(P)\ge 2$, there are at least two elements of $W$ in $I\cup\{e\}$, contrary to the definition of $I$.
\end{proof}

\begin{claim} For each $W\in \cW$, 
$$\sum_{F\in \cH_W}(|F|-k) \le n(n-|X_{2}|) \ldots (n-|X_{d}|).$$
\end{claim}

\begin{proof}[Proof of claim.]
We may assume that $\cH_{W} \neq \emptyset$. Therefore, by Claim \ref{counting_indep}, $a+1 \leq d$. If $r(W)\ge k$ we have $|\cH_W|\le 1$ and, hence, by Claim~\ref{T},
$$
\sum_{F\in \cH_W} |F|  \le n\le |\cT_k| + k|\cH_W|,
$$
as required. So we may assume that $r(W)< k$. Let $a=k-r(W)$.
For a flat $F\in  \cH_W$ we let $\cI_F$ denote the collection of all sets $I\in\cT_{a+1}$ such that $I\cup W$ spans $F$ and $|I\cap \cl(W)|=1$.
Note that for distinct flats $F,F'\in \cH_W$ the sets $\cI_F$ and $\cI_{F'}$ are disjoint. Moreover, for $F\in \cH_W$, we have
$|\cI_F| \geq |F\cap \cl(W)|\times b$, where $b$ is the number of bases of $F\setminus \cl(W)$ in $M\con W$. By basis exchange,
$b\ge |F\setminus \cl(W)|-a+1$. Therefore
\begin{align*}
 |\cI_F| &\ge |F\cap\cl(W)| \times(|F\setminus\cl(W)|-a+1)\\
 &= (|F\cap\cl(W)|-1) \times(|F\setminus\cl(W)|-a) + \\
 &\quad |F\setminus\cl(W)| + (|F\cap\cl(W)|-a+1) -1\\
 &\ge |F|-k.
 \end{align*}
 Hence
$$
\sum_{F\in \cH_W} (|F|-k)  \le |\cT_{a+1}| \le |\cT_d| \leq n(n-|X_{2}|) \ldots (n-|X_{d}|), 
$$
as required.
\end{proof}

Combining the above bounds gives:
\begin{align*}
\sum_{F\in\cF_k} (|F|-k) &\leq \sum_{F\in \cH_0} (|F|-k)+ \sum_{W \in \cW} \sum_{F\in \cH_W} (|F|-k)\\
&\le |\cT_d| +(2^{k(k-1)}-1)n(n-|X_{2}|) \ldots (n-|X_{d}|)  \\
&\le 2^{k(k-1)}n(n-|X_2|)\cdots (n-|X_d|),
 \end{align*}
as required.
\end{proof}

\section{The average size of rank-$k$ flats}

In this section we prove Theorem~\ref{main_flats}. Later, in Section \ref{top-heavy},
we will give an alternative (and shorter) proof by deriving the result as a consequence of Theorem~\ref{thm_flats} and results of Huh and Wang.

We use the following extension of Lund's Theorem.
\begin{theorem} \label{strat3}
For each integer $k\ge 2$, there is a real number $\sigma_k>0$ such that, if
$M$ is a simple, $n$-element, complex-representable matroid with rank at least $2k-1$, and $(X_2,\ldots,X_k)$ is an 
optimal $k$-stratification of $M$, then
$$ |\cF_k(M)| \ge \sigma_k n(n-|X_2|)(n-|X_3|)\cdots (n-|X_k|).$$
\end{theorem}

\begin{proof} By Lemma~\ref{lem:lines}, the result holds for $k=2$. For any given $\kappa\ge 3$ we assume that the result 
holds for $k=\kappa-1$ and prove the result for $k=\kappa$ with
$$ 
\sigma_{\kappa} = \min\left(\frac{\rho_{\kappa}}{\kappa!},\, c_{\kappa}^{-2\kappa},\, 
\frac{\sigma_{\kappa-1}}{(\kappa-1)!c_{\kappa}^2(c_{\kappa}+1)^{\kappa-2}}\right), 
$$
where $\rho_{\kappa}$ and $c_{\kappa}$ are given by Lund's Theorem.

Note that $|X_{\kappa}|=g_{\kappa}(M)$. Therefore,
if $n-|X_{\kappa}|>c_{\kappa}$, then by Lund's Theorem and Lemma~\ref{strat1}, we have
$$ |\cF_k(M)| \ge \frac{\rho_{\kappa}}{\kappa!} n(n-|X_2|)(n-|X_3|)\cdots (n-|X_{\kappa}|),$$
as required. Thus we may assume that $n-|X_{\kappa}|\le c_{\kappa}$; that is
\begin{equation}\label{eq1}
1 \ge \frac {n-|X_{\kappa}|}{c_{\kappa}}.
\end{equation}
We recall that $r(X_k)\le 2(k-1) < r(M)$, so there is a hyperplane $H$ of $M$ that contains $X_k$.
Let $N=M|H$ and $m=|H|$ and let $(Y_2,\ldots,Y_{\kappa-1})$ be an optimal $(\kappa-1)$-stratification in $N$. 
Note that $m\ge n-c_\kappa$ and $r(N) = r(M)-1 > 2(k-1)-1.$
Therefore, by our induction hypothesis and Lemma \ref{strat1},
\begin{align}
 |\cF_{\kappa-1}(N)| &\ge \sigma_{\kappa-1} m(m-|Y_2|)\cdots (m-|Y_{\kappa-1}|)  \\
 &\ge \sigma_{\kappa-1} m(m-g_2(N))\cdots (m-g_{\kappa-1}(N)) \nonumber\\
 &\ge \frac{\sigma_{\kappa-1}}{(\kappa-1)!} m(m-|X_2|)\cdots (m-|X_{\kappa-1}|). \nonumber
\end{align}

Fix an element $e\in E(M)\setminus H$. Since we can extend each flat in $\cF_{\kappa-1}(N)$ to a flat in
$\cF_{\kappa}(M)$ that contains $e$, we have
\begin{equation}\label{eq3}
 |\cF_{\kappa}(M)| \ge|\cF_{\kappa-1}(N)|.
\end{equation}
Let $c=n-|X_{\kappa}|$, thus $c\le c_\kappa$.
For each $i\in\{2,\ldots,\kappa-1\}$ we have $|X_i|\le |X_\kappa|\le n-c$, thus
\begin{align*}
 (c_{\kappa} +1)(m-|X_i|) &\ge (c+1)(m-|X_i|)\\
&\ge n-|X_i| + c (n-|X_i|) - c^2 \\
&\ge n-|X_i|.
\end{align*}
That is
\begin{equation}\label{eq4}
m-|X_i|\ge \frac {n-|X_i|}{c_{\kappa} +1}.
\end{equation}

We may assume that $n\ge c_{\kappa}^2$, since otherwise $\sigma_{\kappa}n^{\kappa}\le 1$ and the result trivially holds. But then 
\begin{equation}\label{eq1}
m \ge \frac n{c_{\kappa}}.
\end{equation}

Therefore, combining all of the above-labelled inequalities
\begin{align*}
|\cF_{\kappa}(M)| &\ge|\cF_{\kappa-1}(N)| \\
&\ge \frac{\sigma_{\kappa-1}}{(\kappa-1)!} m(m-|X_2|)\cdots (m-|X_{\kappa-1}|) \\
&\ge \frac{\sigma_{\kappa-1}}{(\kappa-1)!c_{\kappa}^2(c_{\kappa}+1)^{\kappa-2}}n(n-|X_2|)\cdots (n-|X_{\kappa-1}|)(n-|X_{\kappa}|),
\end{align*}
as required.
\end{proof}

Finally, we can prove Theorem~\ref{main_flats} which we restate here for convenience.
\begin{blah}[Theorem~\ref{main_flats}]
For an integer $k\ge 2$, in a simple complex-representable matroid with rank at least $2k-1$, the average size of a rank-$k$ flat is bounded above by a constant depending only on $k$.
\end{blah}

\begin{proof}
Let $(X_2,\ldots,X_k)$ be an optimal stratification in a simple, $n$-element, complex-representable matroid $M$ with rank at least $2k-1$.
By Theorem~\ref{strat3},
$$ |\cF_k(M)| \ge \sigma_k n(n-|X_2|)(n-|X_3|)\cdots (n-|X_k),$$
where $\sigma_k$ is given by Theorem~\ref{strat3}.
and, by Lemma~\ref{strat2},
$$
 \sum_{F\in\cF_k} (|F|-k) \le 2^{k(k-1)}n(n-|X_2|)\cdots (n-|X_k|).
$$
 Therefore, the average size of a rank-$k$ flat is at most
$$ \frac{2^{k(k-1)}}{\sigma_k} + k ,$$
as required.
\end{proof}

\section{The average size of a flat}

In this section we will prove Theorem~\ref{thm_flats}. For that we require a good estimate on the total number of flats, which we obtain via Lund's Theorem; this result is of independent interest. 
We let $\cF(M)$ denote the set of flats in a matroid $M$.  
\begin{theorem}\label{counting_flats}
For any integer $r\ge 2$, there is a real number $\alpha_r$, such that, for any simple, $n$-element, rank-$r$,  complex-representable matroid $M$ with essential dimension $d$, we have
$$ |\cF(M)|\ge \alpha_r n(n-g_2)\cdots (n-g_d).$$
\end{theorem}

\begin{proof}
Let $c_2,c_3,\ldots,c_d$ and $\rho_2,\ldots,\rho_d$ denote the constants given by Lund's Theorem; we may assume that $d \geq 2$; and by Lemma~\ref{lem:lines} we may assume that $c_2=0$. 
Now let 
$$\alpha_r :=\frac{\min(\rho_2,,\ldots,\rho_d)}{(c_2+1)\cdots(c_d+1)}.$$
Note that $n\ge r> g_2+c_2$. Now choose $\kappa\in\{2,\ldots,k\}$ maximum such that $n>g_{\kappa}+c_{\kappa}.$ By Lund's Theorem we have
$$ |\cF_{\kappa}(M)| \ge \rho_{\kappa} n(n-g_2)(n-g_3)\cdots (n-g_{\kappa}).$$
Therefore
\begin{align*}
|\cF| &\ge |\cF_{\kappa}| \\
&\ge \rho_{\kappa} n(n-g_2)(n-g_3)\cdots (n-g_{\kappa}) \\
&\ge \frac{\rho_{\kappa}}{c_{\kappa+1}\cdots c_d} n(n-g_2)(n-g_3)\cdots (n-g_{d}) \\
&\ge \alpha_r n(n-g_2)(n-g_3)\cdots (n-g_{d}),
\end{align*}
as required.
\end{proof}
The bound in Theorem~\ref{counting_flats} is correct up to a multiplicative error; that is, treating $r$ as a given constant,  $|\cF(M)|$ is $\Theta( n(n-g_1)\cdots (n-g_d))$. We will not prove this explicitly, but it follows from Lund~\cite[Theorem~2]{Lund}.

We now prove Theorem~\ref{thm_flats} which we restate here for convenience.
\begin{blah}[Theorem~\ref{thm_flats}]
For each integer $r\ge 2$,  the average flat-size in a simple rank-$r$ complex-representable matroid is bounded above by a constant depending only upon $r$.
\end{blah}

\begin{proof}
Let $M$ be a simple rank-$r$ complex-representable matroid. By Theorem~\ref{counting_flats},
$$ |\cF(M)|\ge \alpha_r n(n-g_2)\cdots (n-g_d),$$
where $d$ is the essential dimension of $M$ and $\alpha_r$ is the constant given by the theorem.
Now by Lemmas~\ref{strat1} and~\ref{strat2}, for any integer $k\ge 2$, we have
$$
\sum_{F\in\cF_k} (|F|-k)  \le 2^{k(k-1)}(k-1)^k  n(n-g_2)\cdots (n-g_d).
$$
Therefore 
\begin{align*}  \sum_{F\in\cF(M)} (|F|-r) &\le (r+1)2^{r(r-1)}(r-1)^r  n(n-g_2)\cdots (n-g_d)\\
&\le \frac{(r+1)2^{r(r-1)}(r-1)^r}{\alpha_r} |\cF(M)|.
\end{align*}
Hence the average flat size is at most 
$$ \frac{(r+1)2^{r(r-1)}(r-1)^r}{\alpha_r}+r,$$
as required.
\end{proof}

\section{Top heavy matroids}\label{top-heavy}

De Bruijn and Erd\H os~\cite{DeBruijn} proved that, in a rank-$3$ matroid, there are always at least as many lines as points. 
\begin{theorem}[De Bruijn, Erd\H os]\label{DB-E}
The number of lines in a rank-$3$ matroid is always at least the number of points.
\end{theorem}

Dowling and Wilson~\cite{DW} posed the Top Heavy Conjecture as a possible
generalization of the De Bruijn-Erd\H os Theorem. Their conjecture was recently proved Brandon, Huh, Matherne, Proudfoot, and Wang~\cite{BHMPW}; the proof for representable matroids was found earlier by Huh and Wang~\cite{HW} and is considerably shorter.
\begin{theorem}[Top Heavy Conjecture]\label{top-heavy-conjecture}
For any integers $a$, $b$, and $r$ with $1\le a< b\le r-a$, if $M$ is a rank-$r$ matroid, then $|\cF_a(M)|\le |\cF_{b}(M)|$.
\end{theorem}

We say that a rank-$r$ matroid $M$ is {\em top heavy} if $|\cF_{k}|$ is maximized when $k=r-1$.
Thus, for a top heavy matroid, the number of hyperplanes as a fraction of the total number of flats is at least $\frac{1}{r+1}$. Therefore, when $M$ is complex-representable and top-heavy, since the average flat size is bounded, so too is the average size of a hyperplane.
\begin{corollary}\label{top-heavy-average}
For any integer $r\ge 2$, in a top-heavy, complex-representable, rank-$r$ matroid, the average
hyperplane-size is bounded above by a constant depending only on $r$.
\end{corollary}

Lund proved that complex-representable matroids are top heavy unless they are almost degenerate.
\begin{theorem}[Lund] For every integer $r>2$ there is an integer $c_r$ such that,
if $M$ is a simple rank-$r$ complex-representable matroid with $|E(M)|> g_{r-1}(M)+c_r$, then
$M$ is top heavy.
\end{theorem}

Let $M$ be a complex-representable matroid with rank at least $2k-1$. By Theorem~\ref{top-heavy-conjecture}, we have
$$
|\cF_k| \ge |\cF_{k-1}| \ge\cdots\ge |\cF_1|.
$$
That is, the truncation $M'$ of $M$ to rank-$(k+1)$ is top heavy. Then applying Theorem~\ref{top-heavy-average} to $M'$, we see that the average size of a rank-$k$ in $M$ is bounded above by a constant depending only on $k$. This gives an alternative proof of Theorem~\ref{main_flats}.

\section*{Acknowledgements}
We thank Ben Lund, Peter Nelson, and James Oxley for interesting discussions based on an earlier version of this paper. In particular, Ben and James gave counterexamples to the earlier versions of the conjectures and Peter raised some interesting new questions.

\appendix

\section{Matroid terminology}

This appendix borrows liberally from a paper of Geelen and Kroeker~\cite{GeKr}, but is customized to the needs of this paper.

This paper is primarily concerned with ``representable matroids", so we will keep our discussion here to that level of generality. For a more general and thorough introduction to matroid theory, see Oxley~\cite{Oxley}. 

Let $A$ be a matrix, with entries in a field $\bF$, whose set of column-indices is $E$.
We let $\cI$ denote the collection of all subsets $I$ of $E$ such that the columns
of $A$ indexed by $I$ are linearly independent. The {\em matroid represented by $A$} is the pair
$M(A)=(E,\cI)$; the set $E$ is the {\em ground-set} of $M(A)$ and the sets in $\cI$ are the
{\em independent sets}. Thus, the matroid captures the combinatorial information about linear independence without explicitly referring to the specific coordinatization. We say that a matroid is
{\em $\bF$-representable} if it can be represented by a matrix over $\bF$.

It is well-known (and an easy consequence of Hilbert's nullstellensatz) that the matroids representable over any field of characteristic zero are also complex-representable. Therefore the results in this paper apply to any field of characteristic zero.

The definition allows for elements to be represented by the zero-vector; such elements are called {\em loops}. The definition also allows for two elements to be represented by vectors that are scalar multiples of each other; such pairs are said to be in {\em parallel}. 
We call a matroid {\em simple} if it has no loops and no parallel pairs. The matroids coming from geometry are simple, so it suffices for us to consider simple matroids. Nevertheless, loops and parallel pairs do arise in our proofs as a result of ``contraction" (or projection).

The {\em rank} of a set $X\subseteq E$, denoted $r(X)$, is the maximum size of an independent subset of $X$. The {\em rank of $M$}, denoted $r(M)$, is the rank of $E$.  Borrowing terminology from geometry, a {\em rank-$k$ flat} of $M$ is an inclusion-wise maximal subset of $E$ with rank $k$. Thus a {\em point} is a rank-$1$ flat, a {\em line} is a rank-$2$ flat, a {\em plane} is a rank-$3$ flat, and a {\em hyperplane} is a rank-$(r(M)-1)$ flat. The rank-$k$ flats in a matroid correspond to the ``spanned'' $(k-1)$-flats in geometry. 

Note that each set $X\subseteq E$ is contained in a unique flat of rank $r(X)$; we denote that flat by $\cl(X)$ and we refer to that set as the {\em closure} of $X$.
We say that flats $F_1$ and $F_2$ are {\em skew}  if  $r(F_1\cup F_2)=r(F_1)+r(F_2)$. In particular, two lines are skew if their union has rank $4$.

We need the operations of ``restriction" and ``contraction" on matroids. For a set $X\subseteq E$, the {\em restriction} of $M$ to the set $X$ is defined as you would expect and is denoted by $M|X$ or by $M\del (E(M)\setminus X)$.
For a set $C\subseteq E$, we define the {\em contraction} of $C$ in $M$ by $M\con C:= (E\setminus C,\cI')$ where $I\subseteq E\setminus C$ is in $\cI'$ if and only if $r(C\cup I) = r(C)+|I|$. We note that $M\con C$ is a complex-representable matroid; one obtains the representation by ``projecting out" or ``quotienting on" the subspace spanned by $C$. After projection the elements of $C$ become loops so we delete them.  

We also require one further operation, namely ``principal truncation".
Let $F$ be a flat in an $\bF$-representable matroid $M=(E,\cI)$, where $\bF$ is an infinite field.
Let $M^+=(E\cup\{e\},\cI^+)$ be the $\bF$-representable matroid obtained by
adding a new column, indexed by $e$, as freely as possible in the span of $F$ to the representation of $M$. The matroid $M^+$ is uniquely determined by $M$ and $F$. For $X\subseteq E$, we have $r_{M^+}(X\cup \{e\}) = r(M)$ if $F\subseteq\cl(X)$ and $r_{M^+}(X\cup \{e\}) = r(M)+1$ otherwise.
We call $M^+$ the {\em principal extension} of $M$ into $F$, and we refer to $M^+\con e$ as the 
{\em principal truncation} of $F$. In the proof of Theorem~\ref{lund_planes}, we consider the principal truncation of a line; note that the line is reduced to a point by this operation.

\section{Dowling geometries}

Dowling geometries (see~\cite{Dowling}) are a class of matroids with large average flat-size. Before giving the technical definition, we roughly describe the key structural features. Consider a simple rank-$r$ matroid $M$ having a basis $V=\{v_1,\ldots,v_r\}$ such that each element of $M$ is spanned by some pair of elements in $V$ and, for each $1\le i<j\le r$, the pair $\{v_i,v_j\}$ spans a line, say $L_{i,j}$, with at least three points. Every hyperplane of $M$ that does not contain the line $L_{i,j}$ intersects the line in at most one point. Dowling geometries are defined so that every hyperplane hits the line $L_{i,j}$ in at least one point.

Let $G=(S,\times)$ be a finite abelian group, and $r$ be a positive integer.
We define a rank-$r$ matroid $\DG(r,G)$ with ground set
$E=\{v_1,\ldots,v_n\}\cup \{e_{i,j,g}\, :\, 1\le i<j\le n,\, g\in G\}$ as follows.
For each $X\subsetneq \{1,\ldots,n\}$ and sequence $(g_i\, :\, i\in\{1,\ldots,n\}\setminus X)$
the set 
\begin{gather}
\{v_i\,:\, i\in X\}\cup \{e_{i,j,g}\, :\, i,j\in X,\, i<j,\, g\in G\}\notag\\
\cup\{e_{i,j,g}\, :\,1\le i<j\le n,\,i,j\not\in X,\, g_j=g_i\times g\}\notag 
\end{gather}
is a hyperplane, and these are the only hyperplanes.
The size of the hyperplane determined by $X\subsetneq \{1,\ldots,n\}$ and sequence $(g_i\, :\, i\in\{1,\ldots,n\}\setminus X)$ is
$$|X|+|G|{|G|\choose 2} +{n-|G|\choose 2}.$$
There is some redundancy in the description since multiplying the group-elements
$(g_i\, :\, i\in\{1,\ldots,n\}\setminus X)$ by a fixed group-element will give the same 
hyperplane.

The Dowling Geometry is representable over field $\bF$ if $G$ is a subgroup of the multiplicative group of $\bF$. Over the reals, the only non-trivial subgroup is 
$(\{\pm 1\},\times)$. Over the complex-numbers we get any finite cyclic group.
Thus $\DG(r,\bZ_2)$ is real-representable and $\DG(r,\bZ_t)$ is complex-representable 
for all $t\ge 2$.

The lower-bounds on average flat-sizes in Section~2 rely on the following results.
\begin{lemma}\label{weakbounds}
For a fixed integers $r$ and $k$ with $2\le k<r$, the number of rank-$k$ flats in 
$\DG(r,\bZ_t)$ is $\Omega(t^{k})$ and the average size of a rank-$k$ flat is at most 
${k+1\choose 2}+o_t(1)$. Moreover, for $k=r-1$, the size is equal to ${r\choose 2}+o_t(1)$.
\end{lemma}

\begin{proof}
The number of elements in $\DG(r,\bZ^t)$ is $r+t{r\choose 2}$ which is $\Omega(t)$, since we are taking $r$ to be constant. It is an easy application of Do's Theorem that the number of 
rank-$k$ flats is $\Omega(t^k)$. 

The number of $(t+2)$-point lines is ${r\choose 2}$, which, for constant $r$, is $O(1)$. The number of rank-$k$ flats that contains one of these lines is $\Omega(t^{k-2})$ and each of these flats has size $O(t)$, so the sum of the sizes of these flats is $o(t^k)$; thus, for large $t$, these flats do not contribute significantly to the average. 

Any other rank-$k$ flat hits each of the $t+2$-point lines in at most one point, so those flats have size at most ${r \choose 2}$, which is $O(1)$. The number of flats rank-$k$ flats that contain one of
$\{v_1,\ldots,v_r\}$ is $\Omega(t^{k-1})$, so these flats do not contribute significantly to the average.

Consider a rank-$k$ flat $F$ disjoint from $\{v_1,\ldots,v_r\}$.
For $1\le a<b<c\le r$, if $F$ contains a point from two of the lines
$L_{a,b}$, $L_{a,c}$, and $L_{b,c}$, then it also contains a line from the third.
It follows that there is a partition $(X_1,\ldots,X_{n-k+1})$ of $\{1,\ldots,r\}$ into non-empty sets
such that, for $1\le i<j\le r$, the flat $F$ contains a point on $L_{i,j}$ if and only if $i$ and $j$ 
are in the same part of the partition $(X_1,\ldots,X_{n-k})$. Thus 
$$ |F| = {|X_1|\choose 2}+{|X_2|\choose 2}+\cdots+{|X_{n-k}|\choose 2}.$$
This is maximized when all but one of the sets $(X_1,\ldots,X_n-k+1)$ are singletons, in which case $|F|={k+1\choose 2}.$ Moreover, when $k=r-1$ we have $n-k=1$ so $|F|={r\choose 2}.$
\end{proof}

Lemma~\ref{weakbounds} implies that the average flat-size in $\DG(r,\bZ_t)$ is equal to ${r\choose 2}+o_t(1)$.

\begin{thebibliography}{}



\bibitem{Beck}
J. Beck, On the lattice property of the plane and some problems of Dirac, Motzkin and Erdos in combinatorial geometry, {\it{Combinatorica}}, 3 (1983), 281-297. 

\bibitem{BHMPW}
T. Braden, J. Huh. J.P. Matherne, N. Proudfoot, B. Wang,
Singular Hodge theory for combinatorial geometries,
arXiv:2010.06088.

\bibitem{DeBruijn}
N.G. de Bruijn and P. Erd\H os, On a combinatorial problem, {\it{Proceedings of the Section of Sciences of the Koninklijke Nederlandse Akademie van Wetenschappen te Amsterdam}}, 51(10) (1948), 1277-1279.

\bibitem{DeZeeuw}
F. de Zeeuw, Spanned lines and Langer's inequality, arXiv:1802.08015 (2018). 

\bibitem{Do}
T. Do, Extending Erd\H os-Beck's theorem to higher dimensions, {\it{Computational Geometry: Theory and Applications}}, 90 (2020), 101625.

\bibitem{Dowling}
T.A. Dowling, A class of geometric lattices based on finite groups. 
Journal of Combinatorial Theory, Series B, 14 (1973), 61?86.

\bibitem{DW}
T.A. Dowling, R.M. Wilson,
Whitney number inequalities for geometric lattices, Proc. Amer. Math. Soc. 47 (1975), 504?512

\bibitem{GeKr}
J. Geelen, M.E. Kroeker,
A Sylvester-Gallai-type theorem for complex-representable matroids,
arXiv:2310.02826 (2023).

\bibitem{Hirzebruch}
F. Hirzebruch, Arrangements of lines and algebraic surfaces. In Arithmetic and Geometry (Papers dedicated to I. R. Shafarevich), vol. 2, 113-140, Birkhauser (1983).

\bibitem{HW}
J. Huh, B. Wang,
Enumeration of points, lines, planes, etc, Acta Math. 218 (2017), 297?317.

\bibitem{Lovasz}
L. Lov\'asz, Flats in matroids and geometric graphs,
Combinatorial surveys (Proc. Sixth British Combinatorial Conf., Royal Holloway Coll., Egham, 1977), 45-86. Academic Press, London. 

\bibitem{LovYem}
L. Lov\'asz, Y. Yemini, On generic rigidity in the plane,
SIAM J. Alg. Disc. Meth. 3 (1982), 91-98.

\bibitem{Lund}
B. Lund, Essential dimension and the flats spanned by a point set, {\it{Combinatorica}}, 38(5) (2018), 1149-1174.

\bibitem{Melchior}
E. Melchior, Uber vielseite der projektiven ebene, {\it{Deutsche Mathematik}}, 5(1) (1941), 461-475. 

\bibitem{Oxley}
J. Oxley, \textit{Matroid theory, second ed.}, Oxford University Press, New York, (2011).

\bibitem{Pokora}
P. Pokora,
Hirzebruch-type inequalities viewed as tools in combinatorics,
Electronic Journal of Combinatorics 28 (2021), 1-22.

\bibitem{SzemerediTrotter}
E. Szemer\'edi and W.T. Trotter, Extremal problems in discrete geometry, {\it{Combinatorica}}, 3 (1983), 381-392. 

\bibitem{Toth}
C.D. Toth, The Szemer\'edi-Trotter theorem in the complex plane, {\it{Combinatorica}}, 35 (2015), 95-126.

\bibitem{Zahl}
J. Zahl, A Szemer\'edi-Trotter type theorem in R 4, {\it{Discrete \& Computational Geometry}}, 54 (2015), 513-572.  







\end{thebibliography}
\end{document}